\documentclass[11pt, reqno]{amsart}

\usepackage{amssymb,latexsym,amsmath,amsfonts}
\usepackage{enumitem}
\usepackage{mathrsfs}
\usepackage{graphicx}
\usepackage{hyperref}
\usepackage[usenames]{color}

\numberwithin{equation}{section}

\theoremstyle{definition}
\newtheorem{definition}{Definition}[section]
\newtheorem{question}[definition]{Question}

\theoremstyle{remark}
\newtheorem{remark}[definition]{Remark}

\theoremstyle{plain}
\newtheorem{theorem}[definition]{Theorem}
\newtheorem{result}[definition]{Result}
\newtheorem{lemma}[definition]{Lemma}

\newtheorem{example}[definition]{Example}
\newtheorem{corollary}[definition]{Corollary}

\definecolor{DPurple}{rgb}{0.46,0.2,0.69}

\newcommand{\F}{\boldsymbol{\rm F}}
\newcommand{\J}{\boldsymbol{\rm J}}
\newcommand{\projspace}{\mathbb{C}\mathbb{P}^k}

\newcommand{\B}{\mathscr{B}}
\newcommand{\f}{F}
\newcommand{\comple}{\B_{\mu}}
\newcommand\decosum[1]{\sum\nolimits_{{#1}}^{\bullet}}
\newcommand{\supp}{{\rm supp}}
\newcommand{\ind}{{\rm I_2}}
\newcommand{\fr}{\boldsymbol{{\sf f}}}
\newcommand{\D}{D}

\newcommand{\sph}{\widehat{\mathbb{C}}}
\newcommand{\C}{\mathbb{C}} 
\newcommand{\R}{\mathbb{R}}
\newcommand{\Z}{\mathbb{Z}}
\newcommand{\N}{\mathbb{N}}

\begin{document}

\title[Recurrence: meromorphic correspondences \& holomorphic semigroups]{Recurrence 
in the dynamics of meromorphic \\ correspondences and holomorphic semigroups}

\author{Mayuresh Londhe}
\address{Department of Mathematics, Indian Institute of Science, Bangalore 560012, India}
\email{mayureshl@iisc.ac.in}

\begin{abstract}
This paper studies recurrence phenomena in iterative holomorphic dynamics of certain multi-valued maps.
In particular, we prove an analogue of the Poincar{\'e} recurrence theorem for meromorphic 
correspondences with respect to certain dynamically interesting measures associated with them.
Meromorphic correspondences present a significant measure-theoretic obstacle:
the image of a Borel set under a meromorphic correspondence need not be Borel.
We manage this issue using the Measurable Projection Theorem, which is an aspect of descriptive set theory.
We also prove a result concerning invariance properties of the supports of the measures mentioned. 
\end{abstract}


\maketitle

\vspace{-0.25cm}
\section{Introduction and statement of main results}\label{S:intro}

This paper is devoted to the study of recurrence phenomena in iterative holomorphic dynamics beyond the
classical framework of maps. The best-known result on recurrence is the Poincar{\'e} recurrence theorem,
which says, in essence, that for certain self-maps of a probability space, the orbit of a \emph{typical} point of a
measurable subset visits this subset infinitely often. This paper explores how and when this phenomenon,
suitably interpreted, arises for a large class of correspondences.
\smallskip

We begin with a natural situation, in the holomorphic setting, in which the above phenomenon
might be explored. A \emph{rational semigroup} is a semigroup consisting of non-constant rational maps on $\sph$ with
function composition as the semigroup operation. 
Most finitely generated rational semigroups happen to admit
a probability measure that is associated with the semigroup action on $\sph$. The first such construction was
by Boyd \cite{boyd:imfgrs99}. Given such a semigroup and a generating set 
$\mathcal G=\{f_1, \dots, f_N\}$, any \textbf{word} $g$ of the form $g= f_{i_n} \circ \dots \circ f_{i_1}$ is said
to be of \emph{length} $n$, denoted $l(g)=n$. Boyd showed \cite[Theorem~1]{boyd:imfgrs99}
that for a finitely generated semigroup $S$ where $\deg(f)\geq 2$
for every $f\in S$, given any generating set $\mathcal G=\{f_1, \dots, f_N\}$, there exists a Borel probability
measure $\mu_{\mathcal G}$ such that for each $a\in \sph\setminus E(S)$  
\begin{equation}\label{E:boyd}
  \frac{1}{d_{{\mathcal G}}^n} \sum\nolimits_{\substack{g(z)=a \\ l(g)=n}}\!\delta_z \to \mu_{\mathcal G} \; \;
  \text{ as } n \to \infty
\end{equation}
in the weak* topology. Here, $d_{{\mathcal G}}:=\deg(f_1) + \dots + \deg(f_N)$ and $E(S)$ is a set with at
most two points that is independent of $\mathcal G$. 
\smallskip 

The measure constructed above generalizes the measure $\mu_{f}$ constructed by
Freire--Lopes--Ma{\~n}{\'e} \cite{FLM:imrm83},
Lyubich \cite{ljubich:eprers83} for $S=\langle f \rangle$, $\deg(f)\geq 2$. The latter measure is
invariant in the classical sense: i.e., $\mu_{f}(f^{-1}(B))= \mu_{f}(B)$ for every 
Borel subset $B$ of $\sph$. However, in the case of semigroups$\,\neq \langle f \rangle$, the measure
$\mu_{{\mathcal G}}$ does \textbf{not}, in general, possess this invariance. In particular, an important ingredient
in proving the Poincar{\'e} recurrence theorem is lost in the latter case.
This naturally raises the following

\begin{question}\label{Q:analogue}
Is there a form of the Poincar{\'e} recurrence theorem for a rational semigroup having a
set of generators $\mathcal G=\{ f_1, \dots, f_N\}$, where $N\geq 2$?
\end{question}

It turns out that the above question is a version of a broader question that makes sense for a much larger class of
holomorphic dynamical systems. This is because the measure $\mu_{{\mathcal G}}$
is a special case of a type of measure that is preserved, in an appropriate sense, by a class of dynamical systems described
by Result~\ref{R:DS} below. We now proceed to describe these dynamical systems.
Let $X_1$ and $X_2$ be compact complex manifolds of dimension $k$. We shall always assume that manifolds are
connected unless stated otherwise.
A \emph{holomorphic $k$-chain} is a formal linear combination of the form
\begin{equation}\label{E:corresp}
  \Gamma= \sum_{1 \leq i \leq N} m_i\Gamma_i,
\end{equation} 
where the $m_i$'s are positive integers and $\Gamma_i$'s are distinct irreducible complex subvarieties of
$X_1 \times X_2$ of dimension $k$. Let $\pi_s$ be the projection onto $X_s$, $s=1,2$,
and let $|\Gamma|:= \cup_{i=1}^N\Gamma_i$.
We call $\Gamma$ a \emph{meromorphic correspondence of $X_1$ onto $X_2$} if 
$\pi_1|_{\Gamma_i}$ and $\pi_2|_{\Gamma_i}$ are surjective for each $1 \leq i \leq N$. A meromorphic
correspondence $\Gamma$ induces a map $F_{\Gamma}: X_1\to 2^{X_2}$ as follows: 
\[
  F_{\Gamma}(x):= \pi_2 (\pi_1^{-1} \{x\} \cap |\Gamma|).
\]
If $X_1=X_2=X$ then we call $\Gamma$ a meromorphic correspondence \textbf{on} $X$.
If for each $x \in X$ and $1 \leq i \leq N$,
$(\pi_1^{-1}\{x\} \cap \Gamma_i)$ and $(\pi_2^{-1}\{x\} \cap \Gamma_i)$
are finite sets, then we call $\Gamma$ a \emph{holomorphic correspondence on $X$}.
\smallskip

When there is no scope for confusion, we shall, for simplicity of notation, denote $F_{\Gamma}$ by
$F$. This $2^X$-valued map will be the focus of our attention. Also, if no confusion arises, we
shall (as in much of the literature we cite) \emph{refer to the correspondence $\Gamma$ underlying $F$ also as $F$}.
\smallskip

Two meromorphic correspondences on $X$ can be composed with each other\,---\,see
Section~\ref{SS:calculus} for the definition. Keeping in mind the above notational comments,
we shall write $\f^n$ to denote the $n$-fold composition
of a meromorphic correspondence $\f$ on $X$. Thus $\f$ gives rise to a dynamical system on $X$. 
Since, for $x \in X$, $\f(x)$ is not necessarily a singleton, 
one must worry about the ``right''
extension of the notion of recurrence. We define:

\begin{definition}\label{D:rec_general}
Let $\f$ be a meromorphic correspondence on a compact complex manifold $X$ and $A \subseteq X$ be a subset. 
We say that a point $x \in A$ \emph{returns} to $A$ if there exists $n \in \Z_+$ such that
$\f^n(x) \cap A \neq \emptyset$. Also, we say that $x \in A$ is \emph{infinitely recurrent} in $A$ if it returns to $A$
infinitely often, that is, there exists an increasing sequence $\{n_i\}$ of positive integers such that
$\f^{n_i}(x) \cap A \neq \emptyset$ for all $i$.
\end{definition}

The above definition, when specialized to the case of a map, yields the classical notion of
recurrence for maps. With this notion of recurrence, we are closer to formulating a
Poincar{\'e} recurrence theorem in our setting. In the setting of Definition~\ref{D:rec_general},
it is unclear if, in general, $\f$ admits any measure invariant in the classical sense (i.e., that
the measure of the pre-image of $A$ under $F$\,---\,the pre-image can be defined for a meromorphic correspondence
$F$\,---\,is equal to the measure of $A$ for each Borel $A\subseteq X$).
Now, we can \textbf{pull back} a measure $\mu$ by $\f$, denoted by $\f^*\mu$: since the definition
of $\f^*\mu$ is a bit involved, we defer it
to Section~\ref{SS:calculus}. This underlies a notion of invariance for measures called 
\emph{$\f^*$-invariance}: defined by \eqref{E:invariance} below.
Dinh--Sibony \cite{DinhSibony:dvtma06} have identified a rich class of meromorphic correspondences
$\f$ that admit ``nice'' $\f^*$-invariant measures. It turns out that this notion of
$\f^*$-invariant measures is just
the ingredient needed for the following Poincar{\'e} recurrence theorem for meromorphic correspondences:

\begin{theorem}\label{T:poin}
Let $\f$ be a meromorphic correspondence of topological degree $d_t$ on a compact complex manifold
$X$. Suppose there exists
a Borel probability measure $\mu$ on $X$ such that
$\mu$ is $\f^*$-invariant: i.e., it satisfies the condition
\begin{equation}\label{E:invariance}
  \f^*\mu=d_t \mu,
\end{equation}
and suppose $\mu$ does not put any mass on pluripolar sets.
Let $B$ be a Borel subset of $X$ such that $\mu(B) >0$.
Then $\mu$-almost every point of $B$ is infinitely recurrent in $B$. 
\end{theorem}

See Section~\ref{S:lem} for what is meant by a pluripolar subset of a compact
complex manifold.
Note that Theorem~\ref{T:poin} implies the classical Poincar{\'e} recurrence theorem in
the setting of the papers \cite{FLM:imrm83, ljubich:eprers83} with the measure $\mu_f$
(where $f: \sph\to \sph$ is rational). Beyond the $1$-dimensional setting, an immediate question arises:
are the conditions of Theorem~\ref{T:poin} at all satisfied? The following result\,---\,alluded
to above\,---\,reveals that meromorphic correspondences admitting measures that satisfy the 
conditions of Theorem~\ref{T:poin} are abundant:

\begin{result}[Dinh--Sibony, \cite{DinhSibony:dvtma06}]\label{R:DS}
Let $(X, \omega)$ be a compact K\"{a}hler manifold of dimension $k$.
Let $\f$ be a meromorphic correspondence of topological degree $d_t$ on
$(X,\omega)$. Suppose that the dynamical degree of order $k-1$, denoted $d_{k-1}$,
satisfies $d_{k-1} < d_t$. Then, the sequence of measures $\mu_n := d_t^{-n} {(\f^n)}^* {\omega}^k$
($\omega$ normalized so that $\int_X\omega^k=1$)
converges to a Borel probability measure $\mu_F$. Moreover, $\mu_F$ does not put any mass on pluripolar sets
and $\mu_F$ is $\f^*$-invariant.
\end{result}

The dynamical degrees of a correspondence are defined in Section~\ref{S:complex}.
We shall call the measure $\mu_F$ given by Result~\ref{R:DS} the \emph{Dinh--Sibony measure of} $F$.
This gives us an immediate

\begin{corollary}\label{C:poin_for_DS}
Let $(X, \omega)$ be a compact K\"{a}hler manifold of dimension $k$.
Let $\f$ be a meromorphic correspondence of topological degree $d_t$ on
$X$. Suppose that
$d_{k-1} < d_t$. Denote by $\mu_\f$ the Dinh--Sibony measure of $F$. 
Then, for any Borel subset $B\subseteq X$ with $\mu_{\f}(B) >0$,
$\mu_F$-almost every point of $B$ is infinitely recurrent in $B$.
\end{corollary}

One may ask why, apart from its intrinsic interest, one might care for a Poincar{\'e}-type
recurrence theorem for correspondences. To this question, consider the work
\cite{DinhKaufWu:prmdpv20}, in which the authors revisit the problem in
random matrix theory of the asymptotic behaviour of the random products
$s_n\cdots s_1$, $n=1,2,3,\dots,$ where $s_1, s_2, s_3, \dots$ are sampled i.i.d.
from ${\rm SL}_2(\C)$ relative to some probability measure $\nu$.
Issues of recurrence are natural to this problem. At the core of the analysis in
\cite{DinhKaufWu:prmdpv20} is the study of the dynamics of what the authors
call a ``generalized correspondence'' determined by $\nu$. These generalized correspondences
are essential to the results in \cite{DinhKaufWu:prmdpv20}, and questions of recurrence are
approachable through the study of recurrence phenomena for generalized correspondences.
It turns out that for any $\nu$ such that $\supp(\nu)$ is a finite set, the latter object is
a correspondence of the type studied in this work.
\smallskip

We must mention that the proof of Theorem~\ref{T:poin} is not a routine extension of the
argument proving the Poincar{\'e} recurrence theorem. Let us briefly look at the main obstacle.
Let $\f$ be a meromorphic correspondence on a compact complex manifold $X$. Given 
a Borel subset $B\subseteq X$, $\f^\dagger(B)$, where $\f^\dagger$ is the \emph{adjoint} of
$F$ (see Section~\ref{SS:calculus} for a definition), need not necessarily be a Borel subset of $X$. This is in
\textbf{sharp} contrast to the situation
when $F$ is a holomorphic map (in which case $\f^{-1}$ has the role of $\f^{\dagger}$)!
In Section~\ref{S:examples}, we give examples demonstrating this with $X={\mathbb{C}\mathbb{P}^3}$.
In short, the proof of Theorem~\ref{T:poin} requires additional tools.
In our approach, these tools come from descriptive set theory. Roughly speaking, descriptive
set theory is the study of those aspects of measurable sets in which they resemble zero-sets of certain ``nice''
functions. The key result we shall need is the Measureable Projection Theorem
(see Section~\ref{SS:dst}).
\smallskip

We now give an affirmative answer to Question~\ref{Q:analogue}. 
For a rational semigroup $S$, the \emph{Fatou set}, $\F(S)$, 
is the largest open subset of $\sph$ on which $S$ is a normal family. The complement of $\F(S)$ is called the \emph{Julia
set} of $S$, denoted by $\J(S)$.
If $S$ and $\mathcal G$ are as prior to \eqref{E:boyd}, we have the associated measure
$\mu_{\mathcal G}$ given by \eqref{E:boyd}. Boyd showed \cite[Theorem~1]{boyd:imfgrs99}
that $\supp(\mu_{\mathcal G})=\J(S)$. Thus, in the Poincar{\'e} recurrence theorem for this case, it suffices to
consider Borel sets $B \subseteq \J(S)$ (also see Remark~\ref{Re:borel_julia}). In Section~\ref{S:poin},
we shall see how we can use Corollary~\ref{C:poin_for_DS}
to prove the following version of the Poincar{\'e} recurrence theorem for rational semigroups:

\begin{theorem}\label{T:poin_rat}
Let $S$ be a finitely generated rational semigroup such that $\deg(f)\geq 2$
for every $f\in S$. Let $\mathcal G =
\{f_1,\dots, f_N\}$ be a set of generators and $\mu_{\mathcal G}$ be the measure given by \eqref{E:boyd}.
Let $B$ be a Borel subset of $\J(S)$ such that $\mu_{\mathcal G}(B) >0$. Then 
for $\mu_{\mathcal G}$-almost every $x \in B$, there exists an increasing sequence $\{n_i\}$ of positive
integers and words $g_1, g_2, g_3,\dots$ composed of $f_1,\dots, f_N$ such that
\[
  l(g_i)=n_i,  \qquad
  g_{i+1}=h_i\circ g_i \; \ \text{for some word $h_i$ with $l(h_i)=n_{i+1}-n_i$},
\]
and such that $g_{i}(x)\in B$ for each $i \in \Z_+$.
\end{theorem}

We now turn to the subject of invariant sets. For motivation, we present the case of rational
semigroups. Consider a rational semigroup $S$ with at least one element of degree at least 2.
As with the case of the Julia set of a rational map, the Julia set of $S$, $\J(S)$, is
backward invariant. This is established by the following

\begin{result}[Hinkkanen--Martin, \cite{HinkMart:dsrf96}]\label{R:HinkMart}
Let $S$ be a rational semigroup with at least one element of degree\,$\geq 2$. Then $\J(S)$ is backward invariant, i.e., for
each $f \in S$, we have
$f^{-1}(\J(S)) \subseteq \J(S)$.
\end{result}

If $S$ is a finitely generated rational semigroup with each element of degree at least 2
and $\mathcal G=\{f_1, \dots, f_N\}$ is a generating set of $S$
then, for the associated measure $\mu_{\mathcal G}$,
$\supp(\mu_{\mathcal G})= \J(S)$ \cite[Theorem~1]{boyd:imfgrs99}. 
Thus $\supp(\mu_{\mathcal G})$ is backward invariant.
The discussion in Section~\ref{S:poin} linking Theorem~\ref{T:poin_rat} to the formalism of correspondences
suggests an extension of Result~\ref{R:HinkMart} to those meromorphic correspondences featured in
Result~\ref{R:DS}, with the support of the Dinh--Sibony measure in the role of $\J(S)$. For a meromorphic
correspondence $\f$ on $X$,
we say that a subset $A$ of $X$ is \emph{backward invariant} if $\f^\dagger(A) \subseteq A$.
In our next theorem, we prove that the support of
any measure in the larger class of measures associated with $\f$, as featured in Theorem~\ref{T:poin}, is backward
invariant up to a meagre set of points.
On the other hand, the support of any of these measures is, in general, not forward invariant (see
Example~\ref{Ex:support}). 
However,
we also prove that for each $x$ in  the support of a measure as in Theorem~\ref{T:poin}, $\f(x)$ intersects the support.
In short: \textbf{every} point of the support returns (in the sense of Definition~\ref{D:rec_general}) to
the support \textbf{each} time: i.e., for all $n \in \Z_+$.
Before stating our next theorem, we explain some terminology. 
Let $\f$ be a meromorphic correspondence induced by a holomorphic $k$-chain $\Gamma$. We define
the \emph{first} and \emph{second indeterminacy sets of $\f$} as 
\begin{align*}
  {\rm I}_1(\f)&:=\{x \in X : \dim ({\pi_1}^{-1}\{x\} \cap |\Gamma|) > 0\}, \; \text{ and} \\
  \ind(\f)&:=\{x \in X : \dim ({\pi_2}^{-1}\{x\} \cap |\Gamma|) > 0\},
\end{align*}
respectively. We are now ready to state
\begin{theorem}\label{T:recu}
Let $\f$ be a meromorphic correspondence of topological degree $d_t$ on a compact complex manifold
$X$. If $\mu$ is
a Borel probability measure on $X$ such that $\mu$ is $\f^*$-invariant
and $\mu$ does not put any mass on pluripolar sets then the following hold:
\begin{enumerate}[label=$(\alph*)$]
  \item\label{I:forward} If $x \in \supp(\mu)$ then $\f(x) \cap \supp(\mu) \neq \emptyset$.
  \item\label{I:backward} If $y \in \supp(\mu) \setminus \ind(\f)$ then $\f^{\dagger}(y) \subseteq \supp(\mu)$ and if
  $y \in \supp(\mu) \cap \ind(\f)$ then $\f^{\dagger}(y) \cap \supp(\mu) \neq \emptyset$.
\end{enumerate}
Furthermore, for every $x, y \in \supp(\mu)$ and every $n \in \Z_+$,
$\f^n(x) \cap \supp(\mu) \neq \emptyset$ and $(\f^n)^{\dagger}(y) \cap \supp(\mu) \neq \emptyset$.
\end{theorem}

As an application, we give an interesting characterization of the support of the Dinh--Sibony
measure. Given $\f$, we say that a set $A$ is \emph{nearly backward invariant} if
$\f^\dagger(A \setminus \cup_{n=1}^{\infty} \ind(\f^n)) \subseteq A$.

\begin{corollary}\label{C:recu}
Let $(X, \omega)$ be a compact K{\"a}hler manifold.
Let $\f$ be a meromorphic correspondence of topological degree $d_t$ satisfying
$d_{k-1} < d_t$. Denote by $\mu_\f$ the Dinh--Sibony measure of $F$. Then $\supp(\mu_\f)$ is the smallest
closed non-pluripolar nearly backward invariant subset of $X$, by which we mean the following: writing
\[
  \mathscr S := \{ C \subseteq X : C \textnormal{ is closed, non-pluripolar and nearly backward invariant}\},
\]
we have that $\supp(\mu_\f) \in \mathscr S$ and $\supp(\mu_\f) \subseteq C$ for every $C \in \mathscr S$. 
\end{corollary}

One reason the above is interesting is because it has the following immediate
corollary characterizing the support of the Dinh--Sibony measure of a
\textbf{holomorphic} correspondence. As the support of the Dinh--Sibony measure of a rational
map $f$ on $\sph$ with $\deg(f)\geq 2$ (which is just the measure $\mu_f$ above) is its Julia set,
this characterization is reminiscent of one of the characterizations of the Julia set of a rational 
map (see \cite[Theorem~3.1]{McMullen:cdar94}, for instance). The corollary below
is immediate because, for
a holomorphic correspondence $\f$, $\cup_{n=1}^{\infty} \ind(\f^n)=\emptyset$.

\begin{corollary}\label{C:recu_holo}
Let $(X, \omega)$ be a compact K{\"a}hler manifold.
Let $\f$ be a {\bf{\textit{holomorphic}}} correspondence of topological degree $d_t$ satisfying
$d_{k-1} < d_t$, and let $\mu_\f$ denote its Dinh--Sibony measure. Then $\supp(\mu_\f)$ is the smallest
closed non-pluripolar backward invariant subset of $X$.
\end{corollary}
   
\begin{remark}
Stronger versions of the above corollaries
seem to be attainable in view of an assertion (about the ``exceptional set $\mathscr{E}$'') following
Theorem~1.2 in \cite{DinhSibony:dvtma06} wherein the words ``non-pluripolar'' in these corollaries are
replaced by ``not contained in a finite or countable union of proper complex subvarieties of $X$''. 
It is not entirely clear how the latter assertion would follow (except for correspondences of
a special form) merely by adapting the techniques used in references accompanying this claim.
We plan to return to this matter in a forthcoming work.
For this reason, Corollaries~\ref{C:recu} and~\ref{C:recu_holo} appear in the above form.
\end{remark}

\section{Fundamental definitions}\label{S:def}

In this section, we shall collect some definitions and concepts about meromorphic correspondences
that we had mentioned in 
passing in Section~\ref{S:intro}. We shall
also state certain classical definitions and results from descriptive set theory.

\subsection{Calculus of meromorphic correspondences}\label{SS:calculus}

While the objects discussed in this subsection are of somewhat recent
origin, the facts stated here are standard.
We refer the reader to \cite{DinhSibony:dvtma06} for a more detailed discussion.
\smallskip

Let $X$ be a compact complex manifold, let $\Gamma$ be a meromorphic correspondence on $X$, and
let $\f$ and $\Gamma$ be related as described in Section~\ref{S:intro}.
With the presentation of $\Gamma$ as in \eqref{E:corresp},
the coefficient $m_i\in \Z_+$ will be called the \emph{multiplicity} of $\Gamma_i$.
We shall call $\Gamma$ the \emph{graph} of $\f$. 
We define the \emph{support} of the correspondence $\f$ by $|\Gamma|:=\cup_{i=1}^{N} \Gamma_i$.
For $\Gamma_i$ as above, we define
$\Gamma_i^{\dagger}:=\{ (y,x) : (x,y) \in \Gamma_i\}$. Now we use this to define
the \emph{adjoint}
\[
  \Gamma^{\dagger}:= \sum_{1 \leq i \leq N} m_i\Gamma_i^{\dagger}.
\]
The meromorphic correspondence $\Gamma^{\dagger}$ is called the \emph{adjoint} of $\Gamma$.
If $A$ is a subset of $X$ then we define the following set-valued maps
\[
  \f(A):= \pi_2 (\pi_1^{-1} (A) \cap |\Gamma|) {\rm{\ and \ }} \f^{\dagger}(A):= \pi_1 (\pi_2^{-1} (A) \cap |\Gamma|).
\]
For convenience, we have denoted $\f(\{x\})$ and $\f^{\dagger}(\{x\})$ by $\f(x)$ and $\f^{\dagger}(x)$
respectively in Section~\ref{S:intro}. We shall adopt the notational convenience noted in Section~\ref{S:intro}
and refer to the correspondence $\Gamma^{\dagger}$ as $\f^{\dagger}$.
\smallskip

The \emph{topological degree} of $\f$ is the number of points in a generic fiber counted with multiplicities.
It is well known that there exists a non-empty Zariski-open set $\Omega \subseteq X$ such that,
writing $Y^i:= \pi_2^{-1}(\Omega)\cap \Gamma_i$, the
map $\left.\pi_2\right|_{Y^i}: Y^i \to \Omega$ is a $\delta_i$-sheeted holomorphic covering for
some $\delta_i \in \Z_+$, $i=1, \dots , N$. Then the topological degree $d_t(\f)$ of $\f$ is 
\begin{equation}\label{E:topdeg}
  d_t(\f):= \sum_{i=1}^N m_i \delta_i = \sum_{i=1}^N m_i \ \sharp\{y: (y,x) \in \Gamma_i\}, \quad x \in \Omega,
\end{equation}
where $\sharp$ denotes the cardinality. We shall use $d_t$ instead of $d_t(\f)$ whenever there is no confusion.%
\smallskip

Let $\f$ and $\f'$ be two meromorphic correspondences on $X$ induced by holomorphic $k$-chains
\[
  \Gamma = \decosum{1\leq i\leq M}\Gamma_i  {\rm{ \ and \ }}
  \Gamma' = \decosum{1 \leq j \leq M'} \Gamma'_j
\]
respectively. In the above presentation of $\Gamma$ (resp., $\Gamma'$), we do not assume that the
$\Gamma_i$'s (resp., $\Gamma'_j$'s) are distinct varieties\,---\,varieties repeat according to multiplicities.
The ``decorated'' summation above will denote the latter presentation.
Then, by definition, $\f \circ \f'$ is the meromorphic correspondence induced by
\[
  \Gamma \circ \Gamma'= \decosum{1 \leq i \leq M} \decosum{1 \leq j \leq M'}\Gamma_i  \circ \Gamma'_j,
\]
where $\Gamma_i  \circ \Gamma'_j$ is defined as in the following paragraph.
\smallskip

Let $\mathcal A(\f)$ be the smallest complex subvariety of $X$ such that $\pi_l$ resticted to
$|\Gamma| \setminus \pi_l^{-1} (\mathcal A(\f))$ is a holomorphic covering over $X \setminus \mathcal A(\f)$, $l=1,2$.
Let $\mathcal A(\f')$ be defined similarly. Let $\Omega \subseteq X\setminus \mathcal A(\f)$ be a Zariski-open subset
of $X$. Choose a Zariski-open set $\Omega' \subseteq X\setminus \mathcal A(\f')$ satisfying $\f'(\Omega') \subseteq \Omega$.
For example, we can take $\Omega'= (X\setminus \mathcal A(\f')) \setminus 
(\f')^{\dagger}(X \setminus \Omega)$.
Now consider the closure $C_{ij}$ in $X \times X$ of
\begin{equation}\label{E:compos}
  \{(x_1,x_3) \in \Omega' \times X : \exists x_2 \in X {\rm{ \ such \ that \ }}
  (x_1,x_2) \in \Gamma'_j {\rm{ \ and \ }} (x_2,x_3) \in \Gamma_i\}.
\end{equation}
The composition $\Gamma_i  \circ \Gamma'_j$ is the holomorphic $k$-chain whose support is $C_{ij}$
and the multiplicities of whose irreducible components are as follows. Let $C_{ij,\,s}$ denote an arbitrary
irreducible component of $C_{ij}$. Then, its multiplicity in $\Gamma_i  \circ \Gamma'_j$
is the distinct number of $x_2$'s satisfying the condition stated in \eqref{E:compos} for a
\textbf{generic} $(x_1,x_3)\in C_{ij,\,s}$.
We would like to emphasize that $\Gamma_i  \circ \Gamma'_j$ need not be irreducible.
This is the reason why the data defining a meromorphic correspondence must include multiplicities.
Note, furthermore, that the above definition of $\Gamma_i  \circ \Gamma'_j$ is independent of the choice of $\Omega$ and
$\Omega'$. With the above definition of composition, in general, $\f \circ \f'(x)$ need not be same as $\f(\f'(x))$.
But we have

\begin{remark}\label{Re:compo}
If $\f$ and $\f'$ are two meromorphic correspondences on $X$ then 
the following hold:
\begin{enumerate}[label=$(\alph*)$]
  \item \label{I:compo_sub} For every $x \in X$, we have
  \[
    \f \circ \f'(x) \subseteq \f(\f'(x)).
  \]
  \item \label{I:compo_eq} There exists a proper complex subvariety 
    $\mathcal A$ of $X$ such that for every $x \notin \mathcal A$, we have
  \[
    \f \circ \f'(x)= \f(\f'(x)).
  \]
\end{enumerate}
\end{remark}

Let $\D$ be a current on $X$ of bidegree $(p,p)$, $0 \leq p \leq k$. We can pull back $\D$ using the following 
prescription:
\begin{equation}\label{E:pull_current}
  \f^*(\D):= (\pi_1)_*(\pi_2^*\D\wedge [\Gamma])
\end{equation}
whenever the intersection current $(\pi_2^*\D\wedge [\Gamma])$ makes sense. 
Here, $[\Gamma]$ denotes the sum (weighted by multiplicities) of the currents of integration on   
the varieties that constitute $\Gamma$.
In this paper, we are mainly interested in the pull-back of a finite Borel measure\,---\,which can be viewed as
a current of bidegree $(k,k)$. Let $\nu$ be a finite Borel measure on $X$. We will work out
$\f^*\nu$ in detail here. Let $\phi$ be a continuous function on $X$.
\[
  \langle \f^*\nu , \phi \rangle
  =\langle \pi_2^*(\nu)\wedge [\Gamma], \pi_1^*\phi \rangle
  :=\sum_{1 \leq i \leq N} m_i \langle \nu, (\pi_2|_{\Gamma_i})_* (\phi \circ \pi_1) \rangle.
\]
Let $\Omega \subseteq X$ be a Zariski-open set and let $(Y^i, \Omega, \left.\pi_2\right|_{Y^i})$
be the holomorphic coverings introduced prior to \eqref{E:topdeg}
for each $ i=1, \dots , N$. Then for each $x \in \Omega$,
$(\pi_2|_{\Gamma_i})_* (\phi \circ \pi_1)(x)$ is the sum of the values of 
$\phi \circ \pi_1$ on the fiber $\pi_2^{-1}\{x\} \cap \Gamma_i$.
Thus we get
\[
  (\pi_2|_{\Gamma_i})_* (\phi \circ \pi_1)(x)= \sum_{y: (y,x) \in \Gamma_i} \phi(y).
\]
Let us introduce the following notation:
\begin{equation}\label{E:notation}
  T(\phi)(x):= \sum_{i=1}^N m_i \sum_{y: (y,x) \in \Gamma_i} \phi(y), \quad x \in \Omega.
\end{equation}
It is classical that $T(\phi)$ extends continuously to $X \setminus \ind(\f)$ and we define
\[
  \langle \f^*\nu, \phi \rangle := \int_{X \setminus \ind(\f)} T(\phi)(x) d\nu(x).
\]
In view of the Riesz representation theorem, $\f^*\nu$ is a measure, which acts on
continuous functions with the above rule. In particular, if $\nu$ is a Borel probability measure
that puts zero mass on proper complex subvarieties then $\f^*\nu$ is a positive measure of total mass equal to the
topological degree of $\f$. In this case, we also have 
\[
  \langle \f^*\nu, \phi \rangle = \int_{\Omega} T(\phi)(x) d\nu(x),
\]
where $\Omega$ is as above. We shall use this to obtain various inequalities in Section~\ref{S:lem}.
\smallskip

We end this section with a discussion of the pull-back of a measure with respect to the composition of two
meromorphic correspondences. Let $\f$ and $\f'$ be two meromorphic correspondences on $X$. For
any bounded continuous function $\varphi$ defined on a subset of $X$ such that
$X\setminus {\sf dom}(\varphi)$ is a (possibly empty) proper complex subvariety of $X$, let 
$T(\varphi)$, $T'(\varphi)$ and $T''(\varphi)$ denote continuous functions (defined on the largest
possible subsets of ${\sf dom}(\varphi)$) associated\,---\,through the construction above\,---\,with $\f$, $\f'$ and
$\f \circ \f'$, respectively. Let $\phi$ be a continuous function on $X$.
In analogy with Remark~\ref{Re:compo}, we can see that
there exists a complex subvariety $\mathcal A$ of $X$ such that
$T''(\phi) \equiv T(T'(\phi))$ on $X \setminus \mathcal A$. Thus, if the measure $\nu$
puts zero mass on proper complex subvarieties of $X$ then it follows that
\[
  (\f \circ \f')^*\nu = \f'^*(\f^*\nu).
\]
This fact will be crucial in our proof of Theorem~\ref{T:recu}.
\medskip

\subsection{Some descriptive set theory}\label{SS:dst}

The field of descriptive set theory is the study of certain classes of ``well-behaved'' subsets of a
Polish space. The aspect of this theory that concerns us goes back to Suslin's discovery of an error 
in Lebesgue's claim that the projection of a Borel subset of $\R\times\R$
into one the factors is Borel. The result that we need belongs to the theory concerning projections of Borel subsets
of product spaces that arose from Suslin's discovery.
Recall that a \emph{Polish space} is a separable complete metrizable topological space. In particular,
every compact (connected) complex manifold is a Polish space. Before stating a vital result,
we need the following

\begin{definition}\label{D:uni_meas}
Let $(Z, \mathcal F)$ be a measurable space.
A subset $A$ of $Z$ is said to be \emph{universally measurable} if $A$ belongs to the completion of the
$\sigma$-algebra $\mathcal F$ with respect to $\nu$ for every positive finite measure $\nu$ on $(Z, \mathcal F)$.
\end{definition}

We state a result that will be very useful in analysing the structure of the projections of Borel subsets
of $X\times X$ into either of its factors\,---\,for $X$ as in the previous sections. We 
give a version that is stated in the literature (see \cite[Chapter~2]{crauel:rpmPs02}, for instance) for Polish spaces.

\begin{result}[\textsc{Measurable projection theorem}]\label{R:MPT} 
Let $(Z, \mathcal F)$ be a measurable space and $Y$ be a Polish space. Let 
$\B$ be the Borel $\sigma$-algebra of $Y$. Then for every set $A$ in the product $\sigma$-algebra
$\mathcal F \otimes \B$, the projection of $A$ into $Z$
is universally measurable.
\end{result}

\section{Complex analytic preliminaries}\label{S:complex}

This section is devoted to discussing the Dinh--Sibony measure associated with certain meromorphic correspondences, 
which we had mentioned in Section~\ref{S:intro}. We shall also mention a version of the Open Mapping Theorem,
which we use to prove Theorem~\ref{T:recu}. Most of the material in this section is well known; 
the reader familiar with these concepts can safely move on to the next section. 
\smallskip

To discuss the existence of the Dinh--Sibony measure for a meromorphic correspondence, 
we need to define the pull-back of a smooth $(p,p)$-form.
Let $\f$ be a meromorphic correspondence on a compact complex manifold of dimension $k$. Consider
a smooth $(p,p)$-form $\alpha$, $0 \leq p \leq k$.
Since $\alpha$ is also a current of bidegree $(p,p)$, the prescription \eqref{E:pull_current}
defines $\f^*(\alpha)$, since $\pi_2^* \alpha \wedge [\Gamma]$ makes sense as a $(p,p)$-current
on $|\Gamma|$. Indeed, in the case $\Gamma$ is a submanifold, for any $(k-p, k-p)$ form $\Theta$ on
$\Gamma$, one has
\[
  \langle \pi_2^* \alpha \wedge [\Gamma], \Theta\rangle := \int\nolimits_{\Gamma}\Theta\wedge 
  \big(\left. \pi_2\right|_{\Gamma}\big)^*\alpha.
\] 
In the general case, we consider desingularizations of the irreducible components of $|\Gamma|$ to define
$\pi_2^* \alpha \wedge [\Gamma]$. It turns out that the pull-back operation is independent of the desingularizations
chosen.
\smallskip

Consider a compact K\"{a}hler manifold $(X, \omega)$ of dimension $k$, and let $\int {\omega}^k=1$.
Consider a meromorphic correspondence $\f$ on $X$ of topological degree $d_t$.
We define the dynamical degree of order $p$, $0 \leq p \leq k$,
\[
  d_{p}(\f):= \lim_{n \to \infty} {\left( \int_X {(\f^n)}^* {\omega}^{p} \wedge {\omega}^{k-p} \right)}^{1/n}.
\]
Note that $d_k(\f)$ is just the topological degree of $\f$ and 
$d_0(\f)$ is the topological degree of $\f^{\dagger}$.
Thus a meromorphic correspondence $\f$ is a dominant meromorphic map if and only if
$d_0(\f)=1$. We shall use $d_p$ instead of $d_p(\f)$ whenever there is no confusion.
We now define
a sequence $\mu_n:=d_t^{-n} {(\f^n)}^* {\omega}^k$.
Since ${\omega}^k$ is a volume form on $X$, it follows that $\mu_n$ is a sequence
of probability measures.
Under the hypothesis that
$d_{k-1} < d_t$, Dinh--Sibony proved \cite[Section~5]{DinhSibony:dvtma06}
that $\mu_n$ converges in the weak* topology to a $\f^*$-invariant probability measure $\mu_{\f}$.
In fact, they showed that if $u$ is a quasi-p.s.h. function (a function that is locally the sum of a
smooth function and a plurisubharmonic function) then $u$ is $\mu_{\f}$-integrable and
$\langle\mu_n, u\rangle \to \langle\mu_{\f}, u\rangle$ as $n \to \infty$.
This is the mode of convergence implicit in Result~\ref{R:DS}.
In particular, $\mu_{\f}$ puts zero mass on pluripolar sets.
They also showed that preimages of a generic point are equidistributed according to the measure $\mu_{\f}$,
i.e., there exists a pluripolar subset $E$ of $X$ such that for every $a \in X \setminus E$, we have
\[
  d_t^{-n} {(\f^n)}^* {\delta}_a\to \mu_{\f}
\] 
as $n \to \infty$. See \cite[Sections~1 and~5]{DinhSibony:dvtma06} for a detailed discussion.
\smallskip

To state the Open Mapping Theorem alluded to above, we need some terminology and notations. We first
introduce them:
let $X_1$ and $X_2$ be complex manifolds (not necessarily compact) of dimensions $k_1$ and $k_2$, respectively.
Moreover, let $\mathcal A \subseteq X_1$ be a complex subvariety of pure dimension.
Let $f : X_1 \to X_2$ be a holomorphic mapping and $f|_{\mathcal A}$ denote the restriction of $f$ to $\mathcal A$.
Note that, for $z \in \mathcal A$, ${(f|_\mathcal A)}^{-1}  f|_\mathcal A (z) := f^{-1}(f(z)) \cap \mathcal A$ 
is a complex subvariety of $X_1$.
Before stating a result, we define the \emph{rank} of $f|_\mathcal A$ at $z \in \mathcal A$ as 
\[
  {\rm{rank}}_z f|_\mathcal A := {\rm{dim}}_z \mathcal A - {\rm{dim}}_z  {(f|_\mathcal A)}^{-1}  f|_\mathcal A (z).
\]
We are now ready to state

\begin{result}[Remmert, \cite{Remmert:hmaka57}]\label{R:OMT}
Let $X_1$, $X_2$, $\mathcal A$ and $f$ be as above.
Then $f|_{\mathcal A}:\mathcal A \to X_2$ is an open mapping if and only if
${\rm{rank}}_z f|_\mathcal A = {\rm{dim}}(X_2)$ for all $z \in \mathcal A$.
\end{result}

The above result follows from Satz 28 and 29 in \cite{Remmert:hmaka57}\,---\,although Remmert requires $\mathcal A$
to be irreducible, the same conclusion follows if we take $\mathcal A$ to be of pure dimension
(see, for instance, \cite{NagOchi:gftscv90}, Theorem 4.3.4).
\medskip

\section{Examples}\label{S:examples}

Let $X$ be a compact complex manifold and $\f$ a meromorphic correspondence on $X$. 
We had mentioned in Section~\ref{S:intro} that if $B$ is a Borel subset of $X$ then, in general, 
$F^{\dagger}(B)$ is not a Borel subset of $X$.
In this section, we present some examples of meromorphic correspondences
demonstrating that this is a genuine difficulty in proving the results in Section~\ref{S:intro}. Specifically,
we give examples of $\f$ on $\mathbb{C}\mathbb{P}^3$
for which there exists a Borel subset $B$ of $\mathbb{C}\mathbb{P}^3$ such that $F^{\dagger}(B)$
is not a Borel subset of $\mathbb{C}\mathbb{P}^3$. These constructions can easily be generalized
to $\projspace$ for every $k \geq 3$. We also give an example of a correspondence demonstrating
the need for defining recurrence as in Definition~\ref{D:rec_general}.
To present our first two examples, we need a result about bimeasurable functions which we state first.
\smallskip

Let $X$ and $Y$ be topological spaces. Recall that
a function $f:X \to Y$ is Borel measurable if the preimage (under $f$) of every Borel subset of $Y$
is a Borel subset of $X$.
We say that a Borel measurable function $f$ is \emph{bimeasurable} if the image (under $f$) of every Borel subset of $X$ is 
a Borel subset of $Y$.

\begin{result}[Purves, \cite{Purves:bf66}]\label{R:bimeasurable}
Let $X$ and $Y$ be complete separable metric spaces and $E$ a Borel subset of $X$.
Consider a Borel measurable function $f:E \to Y$. 
In order that $f$ is bimeasurable it is necessary and sufficient that the set
$\{\zeta \in Y: f^{-1}{\{\zeta\}} \textnormal{ is uncountable}\}$ is countable. 
\end{result}

For an alternative proof of the above result, also see \cite{Mauldin:bf81}.
\smallskip

If $X$ is as in Section~\ref{S:intro} and $f: X\to X$ is a dominant meromorphic map then, recall, its graph $\Gamma_{f}$
is a meromorphic correspondence on $X$ as defined in Section~\ref{S:intro}. For convenience of
notation, we shall denote the correspondence $\Gamma_{f}$ simply by $f$.

\begin{example}\label{Ex:map}
An example of a dominant meromorphic self-map $g$ on ${\mathbb{C}\mathbb{P}^3}$ such that there exists a Borel subset
$B$ of ${\mathbb{C}\mathbb{P}^3}$ for which $g^\dagger(B)$ is not Borel.
\end{example}
 
\noindent{
We begin with an auxiliary map. 
Consider the meromorphic self-map $f$ on ${\mathbb{C}\mathbb{P}^3}$ given by
\[
f[x:y:z:t]=[x^2+y^2+zt:x^2+yt:xt:t^2].
\]
Note that $f$ is a bimeromorphic map; $f|_{\{[x\,:\,y\,:\,z\,:\,t]\,|\,t \neq 0\}}$ is biholomorphic.
Let ${\rm I_1}(f)$ and ${\rm I_2}(f)$ denote the first and the second indeterminacy sets of $f$ respectively.
It is easy to compute
\[
{\rm I_1}(f)= \{[0:0:1:0]\} \textnormal{ and } {\rm I_2}(f)= \{[x:y:0:0] | (x, y) \in \C^2\setminus \{0\}\}.
\]
Let $\fr$ denote the restriction of $f$ to the open set ${\mathbb{C}\mathbb{P}^3} \setminus {\rm I_1}(f)$.
Since $\fr:{\mathbb{C}\mathbb{P}^3} \setminus {\rm I_1}(f) \to {\mathbb{C}\mathbb{P}^3}$ is holomorphic,
$\fr$ is Borel measurable. 
By definition of ${\rm I_2}(f)$, $f^{\dagger}(\zeta)$ is uncountable for every $\zeta \in {\rm I_2}(f)$.
Observe that $\fr^{-1}\{\zeta\}= f^{\dagger}(\zeta) \setminus {\rm I_1}(f)$ for every
$\zeta \in {\mathbb{C}\mathbb{P}^3}$.
Since ${\rm I_1}(f)$ is a singleton, for every $\zeta \in {\rm I_2}(f)\cap {\sf range}(\fr)$, 
$\fr^{-1}\{\zeta\}$ is an uncountable subset of 
${\mathbb{C}\mathbb{P}^3} \setminus {\rm I_1}(f)$.
It is easy to see that ${\rm I_2}(f)\cap {\sf range}(\fr)$ is uncountable. Thus,
$\{\zeta \in {\mathbb{C}\mathbb{P}^3}: \fr^{-1}{\{\zeta\}} \textnormal{ is uncountable}\}$ is uncountable.
Thus, by Result~\ref{R:bimeasurable}, $\fr$ is not bimeasurable, i.e.,
there exists a Borel subset $B'$ of ${\mathbb{C}\mathbb{P}^3} \setminus {\rm I_1}(f)$
such that $\fr(B')$ is not Borel. It is easy to see that $B'$ is also a Borel subset of ${\mathbb{C}\mathbb{P}^3}$.
Since $B'$ is a subset of ${\mathbb{C}\mathbb{P}^3} \setminus {\rm I_1}(f)$, we get $\fr(B')=f(B')$.
Thus we have a Borel subset $B'$ of ${\mathbb{C}\mathbb{P}^3}$ such that
$f(B')$ is not Borel.
\smallskip

Now, consider  the ``bimeromorphic inverse'' $g$ of $f$
(since $f$ is bimeromorphic, $g$ exists).
As seen above, there exists a Borel subset $B'$ of ${\mathbb{C}\mathbb{P}^3}$ such that 
$f(B')$ is not a Borel set. Since $f(B')= g^\dagger(B')$, $g^\dagger(B')$ is not a Borel set.
Thus, we have an example of a meromorphic map $g$ and a Borel set $B'$ such that
$g^\dagger(B')$ is not a Borel set.
\smallskip

More generally, the above argument works for any bimeromorphic map $f$ on $\projspace$ ($k \geq 3$)
for which ${\rm I_1}(f)$ and ${\rm I_2}(f)$ satisfy
${\rm I_1}(f) \cap f^{\dagger}({\rm I_2}(f))$ is a finite set and $\dim({\rm I_2}(f))>0$.}
\hfill $\blacktriangleleft$
\medskip

With the above example of a meromorphic map at hand, it is now easy to construct an example of a 
meromorphic correspondence that is not a meromorphic map and which maps some Borel set to a non-Borel set. 
In fact, we construct a family of meromorphic correspondences for which $d_t=s^3$ and $d_0=r^3$, where
$r$ and $s$ are arbitrary positive integers.

\begin{example}
An example of a family of meromorphic correspondences $\{\f_{rs}: r,s \in \Z_+ \setminus \{1\}\}$ 
on ${\mathbb{C}\mathbb{P}^3}$, that are not maps, such that each $\f_{rs}$ admits a Borel subset
$B$ of ${\mathbb{C}\mathbb{P}^3}$ such that $\f_{rs}^\dagger(B)$ is not Borel.
\end{example}

\noindent{Consider the family of holomorphic self-maps on ${\mathbb{C}\mathbb{P}^3}$ given by}
\[
h_j[x:y:z:t]=[x^j:y^j:z^j:t^j],
\]
where $j \in \Z_+$. Now, {\bf{fix}} a $j \in \Z_+$.
If $B$ is a Borel set then $(h_j)^{\dagger}(B)=h_j^{-1}(B)$ (since $h_j$ is a holomorphic map) is Borel. 
Let $\mathcal N$ be a non-Borel subset
of ${\mathbb{C}\mathbb{P}^3}$.
Assume that $(h_j)^{\dagger}(\mathcal N)$ is Borel.
Then, by Result~\ref{R:bimeasurable}, $h_j((h_j)^{\dagger}(\mathcal N))=\mathcal N$ is Borel, which contradicts
$\mathcal N$ being non-Borel.
Thus, if $\mathcal N$ is not Borel then $(h_j)^{\dagger}(\mathcal N)$ is not Borel, and this is true for each $j \in \Z_+$.
\smallskip

Fix $r, s \in \Z_+ \setminus \{1\}$.
Consider the meromorphic correspondence
$\f_{rs}:=(h_r)^{\dagger} \circ g \circ h_s$, where $g$ is as in Example~\ref{Ex:map}.
Let $B'$ be a Borel subset of ${\mathbb{C}\mathbb{P}^3}$ such that $g^\dagger(B')$ is not Borel.
Now, consider the Borel set $B'':=(h_r)^{\dagger}(B')$. We claim that $(\f_{rs})^{\dagger}(B'')$ is not a Borel
subset of ${\mathbb{C}\mathbb{P}^3}$. Since $h_r$ and $h_s$ are holomorphic maps,
it follows that
\[
(\f_{rs})^{\dagger}(B'') = (h_s)^{\dagger} (g^{\dagger}(h_r(B''))).
\]
As $h_r$ is a surjective holomorphic map, we have $h_r(B'')=B'$. Recall that
$B'$ is a Borel subset of ${\mathbb{C}\mathbb{P}^3}$ such that $g^{\dagger}(B')$ is not Borel.
By taking $\mathcal N=g^{\dagger}(B')$ in the above discussion, we get that
$(h_s)^{\dagger}(g^{\dagger}(B'))$ is not Borel.
Thus, $(\f_{rs})^{\dagger}(B'')$ is not a Borel
subset of ${\mathbb{C}\mathbb{P}^3}$.
Observe that $d_t(\f_{rs})=s^3$ and $d_0(\f_{rs})=r^3$.
\hfill $\blacktriangleleft$
\medskip

Our last example is concerned with the definition of recurrence, i.e., Definition~\ref{D:rec_general}.
Recall that $x \in A$ is infinitely recurrent in $A$ if 
there exists an increasing sequence $\{n_i\}$ of positive integers such that
$\f^{n_i}(x) \cap A \neq \emptyset$ for all $i$. 
One's first instinct would be to define infinite recurrence as follows:
$x \in A$ is infinitely recurrent in $A$ if 
there exists an increasing sequence $\{n_i\}$ of positive integers such that
$\f^{n_i}(x) \subseteq A$ for all $i$. Here, we give an example which shows that with the
latter definition of recurrence, the Poincar{\'e} recurrence phenomenon would not be true.

\begin{example}\label{Ex:support}
An example of a holomorphic correspondence demonstrating
the need for defining recurrence as in Definition~\ref{D:rec_general}.
\end{example}

\noindent{Consider a finitely generated rational semigroup $S$ generated by
$\mathcal G=\{z^2, z^2 / 2\}$. Let $\J(S)$ denote the Julia set of $S$.} Then (see \cite[Example~1]{boyd:imfgrs99}),
\[
\J(S)=\{z \in \C: 1 \leq |z| \leq 2\}.
\]
Let $\mu_{\mathcal G}$ be the measure given by \eqref{E:boyd}.
It turns out that $\mu_{\mathcal G}$ is the Dinh--Sibony measure associated with the holomorphic 
correspondence $\f_{\mathcal G}$ induced by $\mathcal G$\,---\,see Section~\ref{S:poin}. 
Since $\supp(\mu_{\mathcal G})=\J(S)$ (see \cite[Theorem~1]{boyd:imfgrs99}),
we have $\supp(\mu_{\mathcal G})=\{z \in \C: 1 \leq |z| \leq 2\}$.
Let $f_1(z):=z^2$ and $f_2(z):=z^2/2$.
Now, observe that for $|x|>1$, we have $(f_1)^n(x) \to \infty$ as $n \to \infty$ and
for $|x|<2$, we have $(f_2)^n(x) \to 0$ as $n \to \infty$.
Now, consider, for the Borel set $B$ in Theorem~\ref{T:poin}, $B=\J(S)$.
For any $x \in \J(S)$, we do not have
an increasing sequence $\{n_i\}$ of positive integers such that
$\f_{\mathcal G}^{n_i}(x) \subseteq \J(S)$ for all $i$.
\hfill $\blacktriangleleft$
\medskip

Also, observe in the above example that if $x \in \supp(\mu_{\mathcal G})$ and $|x| \neq \sqrt2$ then
$\f_{\mathcal G}(x) \not\subseteq  \supp(\mu_{\mathcal G})$. 
Thus, Example~\ref{Ex:support} also shows that if $x \in \supp(\mu)$ 
then, in general, $F(x)$ need not be a subset of $\supp(\mu)$ (\emph{cf.} Theorem~\ref{T:recu}).
\medskip

\section{Essential lemmas}\label{S:lem}

This section is devoted to proving certain lemmas that are essential for the proof of Theorems~\ref{T:poin}
and~\ref{T:recu}. Let $X$ be a compact complex manifold.
As discussed in Section~\ref{S:intro}, and demonstrated in Section~\ref{S:examples},
if $B$ is a Borel subset of $X$ then $\f^{\dagger}(B)$ need not be a Borel subset
of $X$. That being said, results from descriptive set theory\,---\,introduced in Section~\ref{SS:dst}\,---\,can be used
to study the measure-theoretic structure of $\f^{\dagger}(B)$. Before doing that let us introduce some notations.
Let $\nu$ be a Borel probability measure on $X$ and let $\B$ denote the Borel $\sigma$-algebra on $X$. 
We use the notation $\B_{\nu}$ to denote the completion of the Borel
$\sigma$-algebra $\B$ with respect to the measure $\nu$.
We first prove an essential lemma.

\begin{lemma}\label{L:proper}
Let $\f$ be a meromorphic correspondence on a compact complex manifold $X$.
Then the following hold:
\begin{enumerate}[label=$(\alph*)$]
  \item \label{I:compact} If $K$ is a compact subset of $X$ then $\f^{\dagger}(K)$ and $\f(K)$ are compact subsets of $X$.
  \item \label{I:proper} If $\mathcal A$ is a {\bf{proper}} complex subvariety of $X$ then 
  $\f^{\dagger}(\mathcal A)$ and $\f(\mathcal A)$ are {\bf{proper}} complex subvarieties of $X$.
  \item \label{I:borel} Let $\nu$ be a Borel probability measure on $X$ and
  $B \in \B$. Then $\f^{\dagger}(B), \f(B) \in \B_{\nu}$.
\end{enumerate}
\end{lemma}

\begin{remark}
The non-elementary part of Lemma~\ref{L:proper} is part~\ref{I:borel}. It is required for the reason discussed at the
beginning of this section, and will be used extensively in the following sections. It is conceivable that the
conclusion of part~\ref{I:borel} is attainable for a measure $\nu$ having the properties stated in
Theorem~\ref{T:poin} without the use of Result~\ref{R:MPT}. However, part~\ref{I:borel} makes a much stronger
assertion since its conclusion holds for \textbf{any} Borel measure $\nu$. We record this stronger assertion
here\,---\,which will also be used in forthcoming work. 
\end{remark}
\begin{proof}
We first prove the parts (a)--(c) above relating to $\f^{\dagger}(A)$, where $A$ denotes subsets of $X$ of various kinds.
As $X$ is compact, $\pi_1$ and $\pi_2$ are proper maps. Thus, by the definition of $\f^{\dagger}(A)$,
$A\subseteq X$, \ref{I:compact} is immediate.
Now assume that $\mathcal A$ is a proper complex subvariety of $X$.
By the proper mapping theorem of Grauert--Remmert, we get that $\f^{\dagger}(\mathcal A)$ is a complex subvariety
of $X$. We also deduce  from it that the codimension of $\f^{\dagger}(\mathcal A)$ in $X$ is strictly positive, since 
$\mathcal A\varsubsetneq X$.   Hence  $\f^{\dagger}(\mathcal A)$ is also a proper
complex subvariety of $X$. We now prove part~\ref{I:borel}.
For $B \in \B$, since $|\Gamma|$ is a closed subset of
$X \times X$, $\pi_2^{-1}(B) \cap |\Gamma|$ is a Borel subset of $X \times X$.
Since $X$ is a manifold, $\B \otimes \B$ is the Borel $\sigma$-algebra of $X \times X$.
Thus $\pi_2^{-1}(B) \cap |\Gamma| \in \B \otimes \B$. Finally, by Result~\ref{R:MPT},
it follows that $\f^{\dagger}(B)$ is universally measurable in $X$.
In particular, by Definition~\ref{D:uni_meas}, we have $\f^{\dagger}(B) \in \B_{\nu}$.
\smallskip

The proofs of  (a)--(c) relating to $\f(A)$ (where $A$ denotes subsets of $X$ of various kinds) are analogous to
those above.
\end{proof}

We also need a result by Vu. Firstly, however, we give a couple of
definitions. Let $X$ be a compact complex manifold.  A function $\varphi: X\to [-\infty, +\infty)$ is said
to be \emph{quasi-p.s.h.} if it can be written locally as the sum of a plurisubharmonic
function and a smooth function. A set $A\varsubsetneq X$ is said to be \emph{pluripolar} if $A$ is
contained in $\{x\in X : \varphi(x) = -\infty\}$ for some quasi-p.s.h. function $\varphi \not \equiv -\infty$ on $X$.
This definition generalizes to compact complex manifolds the definition of a pluripolar set
due to Dinh--Sibony \cite{DinhSibony:dvtma06} in the K{\"a}hler setting. The following result,
which generalizes \cite[Proposition~A.1]{DinhSibony:dvtma06} to compact complex manifolds,
will be used in the proofs of both Theorem~\ref{T:poin} and Theorem~\ref{T:recu}.

\begin{result}[Lemma 2.8 of \cite{Vu:emmskm20}]\label{R:pluripolar}
Every proper complex subvariety $\mathcal A$ of a compact complex manifold $X$ is
a pluripolar set in $X$.
\end{result}

\noindent{(The above result is further generalized in \cite{Vu:lpsp} but the formulation in
Result~\ref{R:pluripolar} suffices for this paper.)}
\smallskip

Let $\mu$ be a measure as in Theorem~\ref{T:poin}. Then the above result
implies that $\mu$ puts zero mass on complex subvarieties of $X$.
Owing to part~\ref{I:compact} of the above lemma, $K$ and $\f^{\dagger}(K)$ are Borel subsets of $X$. We now prove a
lemma comparing their measures for the measure as in Theorem~\ref{T:poin}.

\begin{lemma}\label{L:compact}
Let $\f$ be a meromorphic correspondence on $X$ and
$\mu$ be a measure as in Theorem~\ref{T:poin}. If $K$ is a compact subset of $X$ then
\[
  \mu(K) \leq \mu(\f^{\dagger}(K)).
\]
\end{lemma}

\begin{proof}
Let $\phi$ be a continuous function on $X$. Then $\f^*\mu$ acts on $\phi$ as
\[
  \langle \f^*\mu, \phi \rangle = \int_{X \setminus \ind(\f)} T(\phi)(x) d\mu(x),
\]
where $T(\phi)$ is the associated continuous function on $X \setminus \ind(\f)$ as defined in \eqref{E:notation}.
Let $\Omega$ be a Zariski-open set of the type introduced prior to \eqref{E:topdeg}.
Since $\mu$ puts zero mass on pluripolar sets,
by Result~\ref{R:pluripolar},
$\mu (X \setminus \Omega)=0$. Thus, owing to the definition of $T(\phi)$, we get
\begin{equation}\label{E:pb_formula}
  \langle \f^*\mu, \phi \rangle
  =\int_{\Omega} \Big( \sum_{i=1}^N m_i \sum_{y: (y,x) \in \Gamma_i} \phi(y) \Big)d\mu(x)
  \quad \forall \phi \in {\mathcal C}(X).
\end{equation}
Let $K\subseteq X$ be compact and $\chi_{\f^{\dagger}(K)}$ be the characteristic function of 
$\f^{\dagger}(K)$.
Since this function is upper semicontinuous, there exists a 
decreasing sequence $\{\phi_n\}\subseteq {\mathcal C}(X)$ with $\phi_n \downarrow \chi_{\f^{\dagger}(K)}$.
Recall that $\f^*\mu=d_t \mu$, by hypothesis. Applying this to \eqref{E:pb_formula}, with
$\phi = \phi_n$ therein, we have:
\begin{equation}\label{E:seq_int}
  \frac{1}{d_t} \int_{\Omega} \Big( \sum_{i=1}^N m_i \sum_{y: (y,x) \in \Gamma_i} \phi_n(y) \Big)d\mu(x) 
  =\int_X \phi_n(x) d\mu(x)
  \quad \forall n \in \Z_+.
\end{equation}
Recall that if $x \in \Omega$ then
$d_t= \sum_{i=1}^N m_i \sharp \{y: (y,x) \in \Gamma_i\}$. From this, and the fact that
$\f^{\dagger}(x) \subseteq \f^{\dagger}(K)$ if $x\in K$, we deduce that
\[
  \frac{1}{d_t}\;\sum_{i=1}^N m_i\sum_{y: (y,x) \in \Gamma_i} \phi_n(y) \geq 1 \quad
  \forall x \in K \cap \Omega, \quad \forall n \in \Z_+.
\]
As $\mu(K \setminus \Omega)=0$, this implies that the left-hand side of \eqref{E:seq_int} is
$\geq \mu(K)$ for all $n \in \Z_+$. Now, applying the monotone convergence theorem to the
right-hand side of \eqref{E:seq_int} gives us $\mu(K) \leq \mu(\f^{\dagger}(K))$.
\end{proof}

Let $\f$ be a meromorphic correspondence on $X$ and
$\mu$ be a measure as in Theorem~\ref{T:poin}.
By part ~\ref{I:borel} of Lemma~\ref{L:proper}, we know that $\f^{\dagger}(B) \in \comple$ for 
every Borel subset $B$. We now use Lemma~\ref{L:compact}
to say more about these sets.
We shall use the same symbol $\mu$ to denote the extension of $\mu$ to the $\sigma$-algebra
$\comple$.
We end this section with

\begin{lemma}\label{L:back_ineq}
Let $\f$ be a meromorphic correspondence on $X$ and
$\mu$ be a measure as in Theorem~\ref{T:poin}.
If $A$ is a subset of $X$ such that $A \in \comple$ and $\f^{\dagger}(A) \in \comple$ then
we have
\[
  \mu(A) \leq \mu(\f^{\dagger}(A)).
\]
\end{lemma}

\begin{proof}
Since $X$ is a compact manifold, $\mu$ is a regular measure. Thus we have 
\[
  \mu(A)= \sup \{\mu(K) : K \subseteq A, \ K \rm{\ is\ compact}\}.
\]
If $K$ is a compact subset of $A$ then $\f^{\dagger}(K)$ is a compact subset of $\f^{\dagger}(A)$. Therefore
\[
  \mu(\f^{\dagger}(A))\geq \sup \{\mu(\f^{\dagger}(K)) : K \subseteq A, \ K \rm{\ is\ compact}\}.
\]
By Lemma~\ref{L:compact}, the result follows.
\end{proof}

We shall now prove an analogous result for $\f(A)$.

\begin{lemma}\label{L:forw_ineq}
Let $\f$ be a meromorphic correspondence on $X$ and
$\mu$ be a measure as in Theorem~\ref{T:poin}.
If $A$ is a subset of $X$ such that $A \in \comple$ and $\f(A) \in \comple$ then
we have
\[
  \frac{1}{d_t}\mu(\f(A))  \leq \mu(A) \leq \mu(\f(A)).
\]
\end{lemma}

\begin{proof}
We first prove that $\frac{1}{d_t}\mu(\f(A)) \leq \mu(A)$. Consider an open subset $U$ of $X$. Note that
$\f(U)$ need not be open, but $\f(U) \in \comple$ by Lemma~\ref{L:proper}.
Since the function $\chi_U$ is lower semicontinuous, there exists an
increasing sequence $\{\psi_n\}\subseteq {\mathcal C}(X)$ with $\psi_n \uparrow \chi_{U}$. Since our measure $\mu$
satisfies  $\f^*\mu=d_t \mu$, we get
\begin{equation}\label{E:seq_int2} 
  \frac{1}{d_t} \int_{X \setminus \ind(\f)} \Big( \sum_{i=1}^N m_i \sum_{y: (y,x) \in \Gamma_i} \psi_n(y) \Big)d\mu(x) 
  =\int_X \psi_n(x) d\mu(x)
  \quad \forall n \in \Z_+.
\end{equation}
By monotone convergence theorem, the right-hand side of \eqref{E:seq_int2} converges to
$\mu(U)$ as $n \to \infty$.
With the notation as in \eqref{E:notation} and $\Omega$ therein, it is easy to see that
\[
  \lim_{n \to \infty} T({\psi_n})(x) \geq 1 \quad \forall x \in \f(U) \cap \Omega
\]
and
\[
  \lim_{n \to \infty} T({\psi_n})(x) = 0 \quad \forall x \notin \f(U).
\]
Since $\mu(X \setminus \Omega)=0$ and $T({\psi_n})(x) \leq d_t$ for all $n \in \Z_+$, by dominated convergence theorem,
the limit of the left-hand side of \eqref{E:seq_int2} is $\geq \frac{1}{d_t}\mu(\f(U))$ as $n \to \infty$.
Combining these observations, we obtain $\frac{1}{d_t}\mu(\f(U)) \leq \mu(U)$. Now, if $A$
is a subset of $X$ as in the hypothesis then, by regularity of the measure $\mu$, we have
\begin{align}
  \mu(A)
  &=\inf \{\mu(U) : A \subseteq U,\ U \rm{\ is\ open}\} \notag \\
  &\geq \inf \{\frac{1}{d_t} \mu(\f(U)) : A \subseteq U, \ U \rm{\ is\ open}\} \notag \\
  &\geq \frac{1}{d_t}\mu(\f(A)). \notag 
\end{align}

Now, it remains to show the other inequality.
We show that if $K$ is a compact subset of $X$ then $\mu(K) \leq \mu(\f(K))$. Note that the inequality for
$A$ (as in the hypothesis) follows as in the proof of Lemma~\ref{L:back_ineq}. Here, consider a decreasing 
sequence $\{\phi_n\}\subseteq {\mathcal C}(X)$ such that $\phi_n \downarrow \chi_{K}$. Hereafter,
the proof follows as in the proof of Lemma~\ref{L:compact}.
\end{proof}

\section{The proof of Theorem~\ref{T:poin} and its corollaries}\label{S:poin}


\begin{proof}[The proof of Theorem~\ref{T:poin}]
Let $E$ denote the set of points in $B$ that are infinitely recurrent in $B$.
That is, we have
\[
  E:=\{x \in B : \exists \text{ an increasing sequence $\{n_i\} \in \Z_+$  s.t. }
  \f^{n_i}(x) \cap B \neq \emptyset \ \forall i\}.
\]
Note that $E$ need not belong to $\B$. But if we prove that there exists $E' \subseteq E$ such that
$E' \in \B$ and $\mu(B)=\mu(E')$ then the proof of Theorem~\ref{T:poin} is complete. For $n \in \N$, we define
\[
  E_n:= \bigcup_{j \geq n} (\f^j)^{\dagger}(B).
\]
Note that, in the above expression, we set $(\f^0)^{\dagger}(B):=B$.
Recall that $\comple$ denotes the completion of the Borel $\sigma$-algebra $\B$ with respect to the
measure $\mu$ and
we use the same symbol $\mu$ to denote the extension of the measure $\mu$ to the
$\sigma$-algebra $\comple$.
By Lemma~\ref{L:proper}, 
$(\f^j)^{\dagger}(B) \in \comple$ for all $j \in \N$.
Therefore, $E_n \in \comple$ for all $n \in \N$.
We now give an alternative description of the set $E$ in terms of the sets $E_n$. We shall prove that
\begin{equation}\label{E:alternate}
  E= B \cap \Big(\bigcap_{n \in \N} E_n \Big).
\end{equation}
To prove \eqref{E:alternate}, first assume that $x \in E$. By definition, $x \in B$ and there exists an 
increasing sequence $\{n_i\}$ of positive integers such that $x \in (\f^{n_i})^{\dagger}(B)$ for all $i$ 
and consequently $x \in E_n$ for all $n \in \N$.
Conversly, if $x \in B$ and $x \in E_n$ for all $n \in \N$ then there exists an 
increasing sequence $\{n_i\}$ of positive integers such that $x \in B \cap (F^{n_i})^{\dagger}(B)$.
This implies that $x \in B$ and $\f^{n_i}(x) \cap B \neq \emptyset$ for all $i$.
Thus $x \in B$ is infinitely recurrent in $B$ and we have \eqref{E:alternate}.
\smallskip

We shall use \eqref{E:alternate} to produce the desired Borel set $E' \subseteq E$ such that
$\mu(B)=\mu(E')$. But this needs an intermediate step. Thus we divide the remainder of the proof into two steps.
\medskip

\noindent{\textbf{Step 1.} \emph{Proving, for $E_n$ as above, that
\begin{equation}\label{E:equal}
  \mu(\f^{\dagger}(E_n))= \mu(E_{n+1}) \quad \forall n \in \N.
\end{equation}}}%
\noindent{Note that the above also entails establishing that $\f^{\dagger}(E_n) \in \comple$
for all $n \in \N$. Since we use the same symbol $\mu$ to denote the extension of the measure
$\mu$ to $\comple$, \eqref{E:equal} makes sense. One needs \eqref{E:equal} because, unlike in the case of
compositions of maps,
$\f^{\dagger}(E_n)$ need not be same as $E_{n+1}$ for $n \in \N$.
Let us {\bf {fix}} $n \in \N$ and observe that
\[
  \f^{\dagger}(E_n)=\f^{\dagger} \Big( \bigcup_{j \geq n} (\f^j)^{\dagger}(B) \Big)
  =\bigcup_{j \geq n} \f^{\dagger} ((\f^j)^{\dagger}(B)).
\]
We shall first prove that
$\f^{\dagger}((\f^j)^{\dagger}(B)) \in \comple$ and
$\mu((\f^{j+1})^{\dagger}(B)) = \mu(\f^{\dagger}((\f^j)^{\dagger}(B)))$ for all $j \geq n$.
{\bf{Fix}} $j \geq n$.
By Remark~\ref{Re:compo}, there exists a proper complex subvariety $\mathcal A_{j} \varsubsetneq X$
such that if $x \notin \mathcal A_{j}$ then
$(\f^{{j}+1})^{\dagger}(x) = \f^{\dagger}((\f^{j})^{\dagger}(x))$.
An easy calculation shows 
\begin{align}
  \f^{\dagger}((\f^j)^{\dagger}(B))
  &=\f^{\dagger}((\f^j)^{\dagger}((B\setminus \mathcal A_{j})  \cup (B \cap \mathcal A_{j}))) \notag \\
  &=\f^{\dagger}((\f^j)^{\dagger}(B \setminus \mathcal A_{j}) \cup (\f^j)^{\dagger}(B \cap \mathcal A_{j})) \notag \\
  &=\f^{\dagger}((\f^j)^{\dagger}(B \setminus \mathcal A_{j})) 
  \cup \f^{\dagger}((\f^j)^{\dagger}(B \cap \mathcal A_{j})). \label{E:eq_1}
\end{align}
By the above-mentioned property of points not belonging to $\mathcal A_j$, we have
\begin{equation}\label{E:eq_2}
  \f^{\dagger}((\f^j)^{\dagger}(B \setminus \mathcal A_{j}))
  =\bigcup_{x \in B \setminus \mathcal A_{j}} \f^{\dagger}((\f^{j})^{\dagger}(x))
  =\bigcup_{x \in B \setminus \mathcal A_{j}} (\f^{{j}+1})^{\dagger}(x) 
  =(\f^{j+1})^{\dagger}(B \setminus \mathcal A_{j}).
\end{equation}
Since $B\setminus \mathcal A_{j}$ is a Borel set, by Lemma~\ref{L:proper}, 
$(\f^{j+1})^{\dagger}(B \setminus \mathcal A_{j}) \in \comple$. Thus, by \eqref{E:eq_2}, we get
$\f^{\dagger}((\f^j)^{\dagger}(B \setminus \mathcal A_{j})) \in \comple$.
Again, by Lemma~\ref{L:proper}, $(\f^{{j}+1})^{\dagger}(\mathcal A_{j})$ and
$\f^{\dagger}((\f^{j})^{\dagger}(\mathcal A_{j}))$
are proper complex subvarieties of $X$. Since $\mu$ puts zero mass on pluripolar sets, it follows from Result~\ref{R:pluripolar}
that $\mu((\f^{{j}+1})^{\dagger}(\mathcal A_{j}))= \mu(\f^{\dagger}((\f^{j})^{\dagger}(\mathcal A_{j})))=0$
and consequently, we get
\begin{equation}\label{E:eq_3}
  \mu((\f^{{j}+1})^{\dagger}(B \cap \mathcal A_{j}))= \mu(\f^{\dagger}((\f^{j})^{\dagger}(B \cap \mathcal A_{j})))=0.
\end{equation}
It is an easy consequence of \eqref{E:eq_1} and \eqref{E:eq_3} that $\f^{\dagger}((\f^j)^{\dagger}(B)) \in \comple$.
Now observe that
\begin{equation}\label{E:eq_4}
  (\f^{{j}+1})^{\dagger}(B)
  =(\f^{{j}+1})^{\dagger}(B \setminus \mathcal A_{j}) \cup (\f^{{j}+1})^{\dagger}(B \cap \mathcal A_{j}).
\end{equation}
Owing to \eqref{E:eq_1}, \eqref{E:eq_2}, \eqref{E:eq_3} and \eqref{E:eq_4}, we see that
$\mu((\f^{j+1})^{\dagger}(B)) = \mu(\f^{\dagger}((\f^j)^{\dagger}(B)))$.
Since $j \geq n$ was arbitary, the latter equality holds for all $j \geq n$.
Now, as observed in Remark~\ref{Re:compo},
$(\f^{j+1})^{\dagger}(B) \subseteq \f^{\dagger}((\f^j)^{\dagger}(B))$. By these two relations,  
it is easy to see that $\f^{\dagger}(E_n) \in \comple$ and
$\mu(\f^{\dagger}(E_n))= \mu(E_{n+1})$. Since $n$ was arbitary, this finishes Step 1.}
\pagebreak

\noindent{\textbf{Step 2.} \emph{Construction of a set $E' \subseteq E$ such that $E' \in \B$ and $\mu(B)=\mu(E')$.}}
\vspace{1mm}

\noindent{We have from the definition of $E_n$ that $E_{n+1} \subseteq E_n$ for all $n \in \N$. Recall that $E_n \in \comple$
for all $n \in \N$. Thus it now follows that 
\begin{equation}\label{E:limit_1}
  \lim_{n \to \infty} \mu(E_n) = \mu \big(\bigcap_{n\in \N} E_n\big).
\end{equation}
By Step 1 and since $E_{n+1} \subseteq E_n$ for all $n \in \N$, we get
\[
  \mu(\f^{\dagger}(E_n))=\mu({E_{n+1}})\leq \mu(E_n) \quad \forall n \in \N.
\]
Combining Lemma~\ref{L:back_ineq} with the above expression yields the equality
$\mu({E_n})=\mu({E_{n+1}})$ for all $n \in \N$. Thus it is easy to see that
\begin{equation}\label{E:limit_2}
  \lim_{n \to \infty} \mu(E_n) = \mu (E_0).
\end{equation}
Observe that \eqref{E:alternate} implies that
$B\setminus E = B \setminus \cap_{n\in \N} E_n$. We can see that, by definition, $B=(\f^0)^{\dagger}(B) \subseteq E_0$.
Thus we get $B\setminus E \subseteq  E_0 \setminus \cap_{n\in \N} E_n$. 
Note that, since $E_n \in \comple$ for all $n \in \N$, \eqref{E:alternate} gives us $E \in \comple$.
It is easy to see that
\[
  0 \leq \mu(B\setminus E) \leq \mu \Big(E_0 \setminus \bigcap_{n\in \N} E_n \Big)=\mu(E_0)-
  \mu \Big(\bigcap_{n\in \N} E_n\Big).
\]
Now, by \eqref{E:limit_1} and \eqref{E:limit_2}, we get $\mu(B\setminus E)=0$.
Since $E \in \comple$, it is classical that there exists $E' \subseteq E$
such that $E' \in \B$ and $\mu(E')=\mu(E)=\mu(B)$
and Step 2 is complete.}
\smallskip

As discussed above, the existence of $E'$ given by Step~2 
implies that $\mu$-almost every point of $B$ is infinitely recurrent in $B$.
\end{proof}

The tools introduced above to prove Theorem~\ref{T:poin} are sufficiently robust that using a similar argument
we can prove the following version related to the backward iterates.

\begin{corollary}
Let $\f$ be a meromorphic correspondence of topological degree $d_t$ on a compact complex manifold
$X$. Suppose there exists
a Borel probability measure $\mu$ on $X$ such that
$\mu$ is $\f^*$-invariant
and $\mu$ does not put any mass on pluripolar sets.
Let $B \in \B$ be such that $\mu(B) >0$.
Then $\mu$-almost every point of $B$ is infinitely {\bf{\textit{backward}}} recurrent in $B$,
i.e.,
for $\mu$-almost every $x \in B$, there exists an increasing sequence $\{n_i\}$ of positive integers such that
$(\f^{n_i})^{\dagger}(x) \cap B \neq \emptyset$ for all $i$.
\end{corollary}

We shall not provide a proof here because of the great similarities between it and the proof of Theorem~\ref{T:poin}.
Instead, we just make the following remarks about the proof.

\begin{itemize}[leftmargin=14pt]
  \item We define sets $E$, $E_n$'s analogous to that in the proof of Theorem~\ref{T:poin} by replacing 
  $(\f^j)^{\dagger}$ by $\f^j$ and vice-versa. The proof of 
  $ \mu(\f(E_n))= \mu(E_{n+1}) \ \forall n \in \N$ is routine with appropriate changes to 
  Step 1 in the proof of Theorem~\ref{T:poin}.
  \item The analogue of Step 2 in the proof of Theorem~\ref{T:poin} follows using Lemma~\ref{L:forw_ineq} (as opposed to 
   Lemma~\ref{L:back_ineq}).
\end{itemize}

Now, we shall see how Corollary~\ref{C:poin_for_DS} answers Question~\ref{Q:analogue}. This requires
associating to the data given in Theorem~\ref{T:poin_rat} a holomorphic correspondence on $\sph$
to which Corollary~\ref{C:poin_for_DS} is applicable.
For $\mathcal G=\{f_1, \dots, f_N\}$ as in Theorem~\ref{T:poin_rat}, this correspondence is as follows:
\begin{equation}\label{E:semi_corres}
  \Gamma_{\mathcal G}:= \sum_{1 \leq i \leq N} {\rm{graph}} (f_i).
\end{equation}
The idea of studying the dynamics of a finitely generated rational semigroup through the correspondence
$\Gamma_{\mathcal G}$ was introduced by Bharali--Sridharan in \cite{BhaSri:hcrfgrs17}.

\begin{proof}[The proof of Theorem~\ref{T:poin_rat}]
By hypothesis, $\deg(f_i) \geq 2$ for each $i$.
Consider the holomorphic correspondence $\Gamma_{\mathcal G}$ and note that
$d_t= \deg(f_1) + \dots + \deg(f_N)$ and $d_0= N$. Thus the hypothesis of 
Corollary~\ref{C:poin_for_DS}\,---\,taking $\omega$ to be the (normalized) Fubini--Study form on $\sph$\,---\,is satisfied.
We shall denote $F_{\Gamma_{\mathcal G}}$ simply as $\f_{\mathcal G}$ and adopt the notational
conveniences noted in Section~\ref{S:intro}. Since the hypothesis of Corollary~\ref{C:poin_for_DS} is satisfied,
we get the Dinh--Sibony measure associated with 
$\f_{\mathcal G}$, which we denote by $\mu_{\mathcal G}$. This is same measure as constructed by
Boyd (given by \eqref{E:boyd}) for the semigroup $(S, \mathcal G=\{ f_1, \dots, f_N\})$\,---\,see Section~\ref{S:complex}. 
For any $x \in \sph$ we have
${F_{\mathcal G}}^n(x)= \{g(x): l(g)= n\}$.
Let $B$ be a Borel subset of $\J(S)$
such that $\mu_{\mathcal G}(B)>0$. Then, by Corollary~\ref{C:poin_for_DS},
for almost every $x \in B$, there exists an increasing sequence $\{n'_i\}$ of positive integers 
and a sequence of words $\{e_{i}\}$ (depending on $x$) composed of $f_1,\dots, f_N$,
such that $e_{i}(x) \in B$, where $l(e_i)=n'_i$. Note that there need not exist,
for any $i \in \Z_+$, a word $h_i$ such that $e_{i+1}=h_i \circ e_i$.
\smallskip

Now, choose $f_{i_1} \in \{f_1, \dots, f_N\}$ such that infinitely many terms in the sequence of words $\{e_{i}\}$
start with $f_{i_1}$, i.e., $e= h \circ f_{i_1}$ for some word $h$.
Further, choose $f_{i_2} \in \{f_1, \dots, f_N\}$ such that infinitely many words 
in $\{e_{i}\}$ start with $f_{i_2} \circ f_{i_1}$. Inductively, we similarly choose  
$f_{i_j}$ having chosen $f_{i_1}, \dots, f_{i_{j-1}}$. We are interested in the sequence of words
$\{f_{i_j} \circ \dots \circ f_{i_1}\}_{j \in \Z_+}$. 
By the choice of $f_{i_j}$ for each $j \in \Z_+$, it follows that the sequences of words $\{e_{i}\}$ 
and $\{f_{i_j} \circ \dots \circ f_{i_1}\}$ share infinitely many terms.
Let $\{g_{i}\}$ be a subsequence of both sequences
$\{e_{i}\}$ and $\{f_{i_j} \circ \dots \circ f_{i_1}\}$. As a subsequence of $\{e_{i}\}$, we have,
for each $i \in \Z_+$, $g_{i}(x) \in B$ and $l(g_i)=n_i$ for some increasing sequence $\{n_i\}$ 
of positive integers. As a subsequence of $\{f_{i_j} \circ \dots \circ f_{i_1}\}$, we have
$g_{i+1}=h_i \circ g_i$ for some word $h_i$ for each $i \in \Z_+$; obviously, $l(h_i)=n_{i+1}-n_i$.
\end{proof}

\begin{remark}\label{Re:borel_julia}
We must note here that the focus on $B \subseteq \J(S)$ in Theorem~\ref{T:poin_rat} is meant to
make the latter more informative, taking advantage of \cite[Theorem~1]{boyd:imfgrs99}.
In the \textbf{proof} of Theorem~\ref{T:poin_rat}, $\J(S)$ plays no key role. This suggests that
the latter theorem is generalizable to higher dimensions.
\end{remark}

Given a set of non-constant holomorphic self-maps $\{ f_1, \dots, f_N\}=:\mathcal G$ of $\projspace$, we can
define the holomorphic
correspondence $\Gamma_{\mathcal G}$ in analogy with \eqref{E:semi_corres}.

\begin{theorem}\label{T:coro_semi}
Consider the semigroup generated by 
$f_1, \dots, f_N$, non-constant holomorphic self-maps of $\projspace$, for some $k\in \Z_+$,
with at least one map of (algebraic) degree at least 2.
Denote by $\mu_{\mathcal G}$ the Dinh--Sibony measure associated with the holomorphic correspondence 
$\Gamma_{\mathcal G}$,
where $\mathcal G= \{ f_1, \dots, f_N\}$.
Let $B$ be a Borel subset of $\projspace$ such that $\mu_{\mathcal G}(B) >0$. Then 
for $\mu_{\mathcal G}$-almost every $x \in B$, there exists an increasing sequence $\{n_i\}$ of positive integers
and words $g_1, g_2, g_3,\dots$ composed of $f_1,\dots, f_N$ such that
$l(g_i)=n_i$,\; $g_{i+1}=h_i\circ g_i$ for some word $h_i$ with $l(h_i)=n_{i+1}-n_i$,
and such that $g_{i}(x)\in B$ for each $i \in \Z_+$.
\end{theorem}

We do not know whether $\supp(\mu_{\mathcal G})$ is a subset of the Julia set of the above semigroup 
when $k \geq 2$. It is well known that the topological degree of each $f_i$ equals ${\deg(f_i)}^k$
(see, for instance, \cite[Section~4]{RussShiff:vdsrmcd97}), where $\deg(f_i)$ denotes the algebraic degree.
Thus, for $\Gamma_{\mathcal G}$, $d_t=\deg(f_1)^k + \dots + \deg(f_N)^k$. It turns out that 
$d_{k-1}=\deg(f_1)^{k-1} + \dots + \deg(f_N)^{k-1} < d_t$
(see \cite[Section~4]{BhaSri:ehcecrs20} for an easy computation). 
With this observation, and given Remark~\ref{Re:borel_julia}, the proof of Theorem~\ref{T:coro_semi}
is same as the proof of Theorem~\ref{T:poin_rat}.
\medskip

\section{The proof of Theorem~\ref{T:recu}}\label{S:recu}

In the proof of Theorem~\ref{T:recu}, we need continuity-like properties of the relation associated with a correspondence.
Thus, before beginning the proof, we recall some definitions and results about relations.
A \emph{relation} on a set $X$ is a subset of $X \times X$.
Let $R$ be a relation on $X$ and let $A \subseteq X$. Then we define
\[
  R(A):= \{ y \in X : \exists x \in A {\rm{ \ satisfying \ }} (x,y) \in R\}.
\]  
It is easy to see that $R(A)= \pi_2(\pi_1^{-1}(A) \cap R)$, where $\pi_s$'s denote the projections on respective coordinates.
A relation $R$ on a compact Hausdorff space $X$ is called \emph{closed} if $R$ is a closed subset of
$X \times X$. We are now ready to state a result related to the continuity-like properties of a closed relation: 

\begin{result}[Theorem 2.3 of \cite{McGehee:acrchs92}]\label{R:relation}
Let $R$ be a closed relation on a compact Hausdorff space $X$. If $K$ is a compact subset of $X$ and if $U$ is an 
open set containing $R(K)$ then there exists an open set $V$ containing $K$ such that
$R(V) \subseteq U$. 
\end{result}

Let us now return to the case of our interest: let $X$ be a compact manifold of dimension $k$ and
$\f$ be the meromorphic correspondence induced by a holomorphic $k$-chain $\Gamma$ on $X$. Note that
$|\Gamma|$ is a closed subset of $X \times X$. Thus, if $K$ and $U$ are as in the above result then there is
an open set $V$ containing $K$ such that $\f(V) \subseteq U$. 

\begin{proof}[The proof of Theorem~\ref{T:recu}]
We begin with the proof of part~\ref{I:forward}.
Let $x \in \supp(\mu)$. Assume for a contradiction that $\f(x) \cap \supp(\mu)=\emptyset$.
By part~\ref{I:compact} of
Lemma~\ref{L:proper}, $\f(x)$ is a compact subset of $X$. Since $\supp(\mu)$ is a closed set,
there is an open set $U$ containing $\f(x)$ such that
\[
  U \cap \supp(\mu)= \emptyset.
\]
Thus $\mu(U)=0$. Owing to Result~\ref{R:relation}, we get an open set $V$ containing $x$ satisfying $\f(V) \subseteq U$.
Thus we have $\mu(\f(V))=0$. By Lemma~\ref{L:forw_ineq}, it now follows that $\mu(V)=0$. This contradicts the assumption
that $x \in \supp(\mu)$. Thus $\f(x) \cap \supp(\mu) \neq \emptyset$.
\smallskip

We now prove part~\ref{I:backward}.
Let $y \in \supp(\mu) \setminus \ind(\f)$. Assume that there exists $x \in \f^{\dagger}(y)$
such that $x \notin \supp(\mu)$. Thus there exists an open set $V$ containing $x$ such that
$\mu(V)=0$. 
Note that $\pi_1^{-1}(V) \cap |\Gamma|$ is an open subset of $|\Gamma|$ and $(x,y) \in \pi_1^{-1}(V) \cap |\Gamma|$.
Also note that since $y \notin \ind(\f)$, $(x,y) \notin \pi_2^{-1}(\ind(\f)) \cap |\Gamma|$. We shall now invoke 
Result~\ref{R:OMT}: take $X_1:= (X \times X) \setminus \pi_2^{-1} (\ind(\f))$, 
$X_2:=X$, $\mathcal A:= |\Gamma| \setminus \pi_2^{-1} (\ind(\f))$ and $f:= \pi_2$.
By the definition of $\ind(\f)$, for every $z \in \mathcal A$,  ${\rm{dim}}_z  {(f|_\mathcal A)}^{-1}  f|_\mathcal A (z)=0$.
Since $\Gamma$ is of pure dimension $k$, we have ${\rm{rank}}_z f|_\mathcal A=k={\rm{dim}}(X)$ for every $z \in \mathcal A$.
Thus $\pi_2|_{|\Gamma| \setminus \pi_2^{-1} (\ind(\f))}$ is an open mapping. It now follows easily that
$\f(V)= \pi_2 (\pi_1^{-1}(V) \cap |\Gamma|)$ contains $y$ as an interior point.
By Lemma~\ref{L:forw_ineq},
\[
  \frac{1}{d_t}\mu(\f(V))  \leq \mu(V).
\]
Thus, since $\mu(V)=0$, $\mu(\f(V)) =0$. As $y$ is an interior point of $\f(V)$, we get that
$y \notin \supp(\mu)$, resulting in a contradiction. Thus $\f^{\dagger}(y) \subseteq \supp(\mu)$.
What remains to be shown in part~\ref{I:backward} follows along the lines of the argument of part~\ref{I:forward},
using Lemma~\ref{L:back_ineq} instead of Lemma~\ref{L:forw_ineq}.
\smallskip

From the previous paragraphs, we already have
\[
  \underbrace{\f\big( \dots \f(\f(x)) \dots \big)}_{n-\rm{times}} \bigcap \supp(\mu) \neq \emptyset \ {\rm{and}} \ 
  \underbrace{\f^{\dagger}\big( \dots \f^{\dagger}(\f^{\dagger}(y)) \dots \big)}_{n-\rm{times}} 
  \bigcap \supp(\mu) \neq \emptyset
\]
for all $x, y \in \supp(\mu)$ and for all $n \in \Z_+$.
But this is not enough to conclude the proof of Theorem~\ref{T:recu}\,---\,see Remark~\ref{Re:compo}.
To remedy this, we first consider an iterate of $\f$ and then invoke the argument as in 
part~\ref{I:forward} and part~\ref{I:backward} above.
Since $\mu$ is an $\f^*$-invariant measure and does not put any mass on pluripolar sets, 
owing to the discussion in the last paragraph of Section~\ref{SS:calculus}, 
$ (\f^{n+1})^*\mu= \f^*(\f^{n})^*\mu$ for every $n \in \Z_+$. Inductively, it follows that
$\mu$ is an $(\f^n)^*$-invariant measure for every $n \in \Z_+$, i.e.,
\[
  (\f^n)^*\mu=(d_t)^n \mu \quad \forall n \in \Z_+.
\]
Consequently, we get analogous inequalities as in Lemma~\ref{L:back_ineq} and Lemma~\ref{L:forw_ineq}.
Now the remainder of the proof proceeds along the lines of the proof of part~\ref{I:forward} 
and part~\ref{I:backward} above.
\end{proof}

Recall that a set $A$ is nearly backward invariant if
$\f^\dagger(A \setminus \cup_{n=1}^{\infty} \ind(\f^n)) \subseteq A$. We now give

\begin{proof}[The proof of Corollary~\ref{C:recu}]
It follows from part~\ref{I:backward} of Theorem~\ref{T:recu} that $\supp(\mu_\f)$ satisfies
$\f^\dagger( \supp(\mu_\f)\setminus \cup_{n=1}^{\infty} \ind(\f^n)) \subseteq \supp(\mu_\f)$.
Recall that $\mu_{\f}$ puts zero mass on pluripolar sets.
This implies that $\supp(\mu_{\f})$ is a non-pluripolar set.
Since $\supp(\mu_{\f})$ is a closed set, we get $\supp(\mu_\f) \in \mathscr S$.
Now, consider $C \in \mathscr S$.
Since $C$ is a closed non-pluripolar set and $\cup_{n=1}^{\infty} \ind(\f^n)$ is pluripolar,
there exists $a \in C\setminus \cup_{n=1}^{\infty} \ind(\f^n)$ (see Section~\ref{S:complex}) such that
\[
  d_t^{-n} {(\f^n)}^* {\delta}_a\to \mu_{\f}
\] 
as $n \to \infty$.
From this it is easy to see, owing to $C$ being nearly backward invariant and
closed, that $\supp(\mu_{\f}) \subseteq C$.
Since $C \in \mathscr S$ is arbitrary, we conclude that
$\supp(\mu_\f)$ is the smallest closed non-pluripolar nearly backward invariant subset of $X$.
\end{proof}

\section*{Acknowledgments}
\noindent{I would like to thank my thesis adviser, Prof. Gautam Bharali, 
for several fruitful discussions
and for his support
during the course of this work.
This work is supported by a scholarship from the National Board for Higher Mathematics
(Ref. No. 2/39(2)/2016/NBHM/R\&D-II/11411) and by a UGC CAS-II grant (Grant No. F.510/25/CAS-II/2018(SAP-I)).}

\end{document}